\documentclass[11pt,a4paper,reqno]{amsart}
\usepackage[usenames]{color}
\usepackage{amsmath, amsthm, amscd, amsfonts, amssymb}

\newtheorem{theorem}{Theorem}[section]
\newtheorem{lemma}[theorem]{Lemma}
\newtheorem{proposition}[theorem]{Proposition}
\newtheorem{corollary}[theorem]{Corollary}
\theoremstyle{definition}
\newtheorem{definition}[theorem]{Definition}

\numberwithin{equation}{section}

\begin{document}
\noindent {\footnotesize\tiny}\\[1.00in]

\textcolor[rgb]{0.00,0.00,1.00}{}
\title[]{Proximal Point Algorithm for Quasi-convex Minimization Problems in metric spaces}
\maketitle
\begin{center}
	{\sf Hadi Khat\texttt{}ibzadeh\footnote{E-mail: $^{1}$hkhatibzadeh@znu.ac.ir, $^{2}$mohebbi@znu.ac.ir.} and Vahid Mohebbi$^2$}\\
	{\footnotesize{\it $^{1,2}$ Department of Mathematics, University
			of Zanjan, P. O. Box 45195-313, Zanjan, Iran.}}
\end{center}
\begin{abstract}
In this paper, the proximal point algorithm for quasi-convex minimization problem in nonpositive curvature metric spaces is studied. We prove $\Delta$-convergence of the generated sequence to a critical point (which is defined in the text) of an objective convex, proper and lower semicontinuous function with at least a minimum point as well as some strong convergence results to a minimum point with some additional conditions. The results extend the recent results of the proximal point algorithm in Hadamard manifolds and CAT(0) spaces.
\end{abstract}

\section{Introduction}
 Convex functions and their generalizations (pseudo-convex and quasi-convex functions) have founded several applications in optimization and economics because of their nice minimization properties. These concepts are traditionally
defined in linear spaces. But they are extendable in some geodesical spaces like Riemannian manifolds and nonpositive curvature metric spaces by means of Alexandroff, which are nonlinear version of Hilbert spaces.

In some of constrained minimization problems the objective function may not be convex or even quasi-convex and the constraint set is not a linear space, but the objective function may be convex or quasi-convex along geodesics of the constraint set as a submanifold of the linear space. Then the non-convex and constrained minimization problem can be change to a non-constrained and (quasi)convex minimization problem. In \cite{ah-kh2} the reader can see an example.

A popular method in convex minimization is the proximal point algorithm which introduced and improved by Martinet \cite{mar}, Rockafellar \cite{ro} and the others. As an advantage of this method, it is extendable to nonlinear spaces, like Riemannian manifolds and metric spaces of nonpositive curvature. Ferreira and Oliveira \cite{fer-oli}, Li, Lopez and Martin-Marquez \cite{llm} and Ahmadi and the first author \cite{ah-kh1} studied the proximal point algorithm in Hadamard manifolds. Bacak \cite{bac1} studied the proximal point method in nonpositive curvature metric spaces. He proved $\Delta$-convergence of the algorithm to a minimum point of the convex function. The proximal point algorithm also has been used for minimization of a quasi-convex function in Hadamard manifolds in \cite{bfo, pqo, tzh, tx}. The main goal of this paper is to study the proximal point algorithm for quasi-convex functions in Hadamard spaces framework, which extends the previous results in the literature. In Section 2, we introduce Hadamard spaces and quasi-linearization of Berg and Nikolaev as well as $\Delta$-convergence introduced by Lim \cite{lim} in Hadamard spaces as a similar argument of weak convergence in Hilbert spaces. In the sequel we introduce some necessary definitions of generalized convex functions and their basic properties. Section 3 is devoted to the proximal point algorithm and $\Delta$-convergence of its generated sequence to a critical point of a quasi-convex function with at least a minimum point. We also establish the strong convergence with additional assumptions.

\section{Preliminaries and Basic Facts}

Let $(X,d)$ be a metric space. A geodesic from $x$ to $y$ is a map
$\gamma$ from the closed interval $[0,d(x,y)] \subset \mathbb{R}$
to $X$ such that $\gamma(0) = x,\ \gamma(d(x,y)) = y$ and
$d(\gamma(t), \gamma(t')) = |t - t'| $ for all $t, t'\in [0,d(x,y)]$.
The space $(X, d)$ is said to be a
geodesic space if every two points of X are joined by a geodesic. The metric segment $[x,y]$ contains the images of all geodesics, which connect $x$ to $y$. $X$ is called unique geodesic iff $[x,y]$ contains only one geodesic.\\
Let $X$ be a unique geodesic metric space. For each $x, y\in X$ and for each $t\in [0, 1]$, there exists a unique point $z\in [x, y]$ such that
$d(x, z) = td(x, y)$ and $d(y, z) = (1 - t)d(x, y)$. We will use the notation $(1 - t)x \oplus ty$ for the unique point $z$ satisfying the above statement.

In a unique geodesic metric space $X$, a set $A\subset X$ is called convex iff for each $x,y\in A$, $[x,y]\subset A$. A function $f:X\rightarrow]-\infty,+\infty]$ is called\\ (i) convex
iff
\begin{center}
 $f(\lambda x\oplus(1-\lambda)y)\leq\lambda f(x)+(1-\lambda)f(y),\ \ \forall x,y\in X\ \text{and} \ \forall \  0\leq\lambda\leq1$
\end{center}
(ii) strictly convex iff
\begin{center}
$f(\lambda x+(1-\lambda)y)<\lambda f(x)+(1-\lambda)f(y),\ \ \forall x,y\in X, x\neq y \ \text{and}\ \ \forall \ 0\leq\lambda\leq1$
\end{center}
(iii) $\alpha$-weakly convex for some $\alpha>0$ iff
\begin{center}
$f(\lambda x\oplus(1-\lambda)y)\leq\lambda f(x)+(1-\lambda)f(y)+\alpha\lambda(1-\lambda)d^2(x,y),\ \ \forall x,y\in~X\ \text{and} \ \forall \ 0\leq\lambda\leq1$
\end{center}
(iv) $\alpha$-strongly convex for some $\alpha>0$ iff
\begin{center}
$f(\lambda x\oplus(1-\lambda)y)\leq\lambda f(x)+(1-\lambda)f(y)-\alpha\lambda(1-\lambda)d^2(x,y),\ \ \forall x,y\in~X\ \text{and} \ \forall \ 0\leq\lambda\leq1$
\end{center}
(v) quasi convex iff
\begin{center}
$f(\lambda x\oplus(1-\lambda)y)\leq\max\{ f(x),f(y)\},\ \ \forall x,y\in X\ \text{and} \ \forall \ 0\leq\lambda\leq1$
\end{center} 
equivalently, for each $r\in \mathbb{R}$, the sub-level set $L^f_r:=\{x\in X:\ \ f(x)\leq r\}$ is a convex subset of $X$.\\
(vi) $\alpha$-strongly quasi-convex for some $\alpha>0$ iff
\begin{center}
$f(\lambda x\oplus(1-\lambda)y)\leq\max\{ f(x),f(y)\}-\alpha\lambda(1-\lambda)d^2(x,y),\ \ \forall x,y\in~X\ \text{and} \ \forall \ 0~\leq~\lambda\leq1$
\end{center}
for each $x,y\in X$ and $0<\lambda<1$.\\
(vii) pseudo-convex  iff\\
$f(y)>f(x)$ implies that there exist $\beta(x,y)>0$  and $0<\delta(x,y)\leq1$  such that $f(y)-f(tx\oplus(1-t)y)\geq t\beta(x,y)$,  $\forall t\in(0,\delta(x,y))$.\\
A unique geodesic space $X$ is called CAT(0) space if for all $x\in X$ the mapping $d^2(x,\cdot):X\rightarrow\mathbb{R}$ is 1-strongly convex. A complete CAT(0) space is called a Hadamard space.

Berg and Nikolaev in \cite{ber-nik1, ber-nik2} introduced the concept of quasi-linearization
along these lines. Let us formally denote a pair $(a, b) \in X \times  X$ by $\overset {\rightarrow}{ab}$ and call it a vector. Then quasi-linearization is defined as a map $\langle\cdot,\cdot\rangle : (X \times X)\times(X \times X)\rightarrow \mathbb{R}$ defined by
\begin{center}
$\langle\overset {\rightarrow}{ab},\overset {\rightarrow}{cd}\rangle
=\frac{1}{2}\{d^2(a, d) + d^2(b, c) - d^2(a, c) - d^2(b, d)\}  \   \   \    \ (a, b, c, d \in X).$
\end{center}
It is easily seen that
$\langle\overset {\rightarrow}{ab},\overset {\rightarrow}{cd}\rangle
=
\langle\overset {\rightarrow}{cd},\overset {\rightarrow}{ab}\rangle$,
$\langle\overset {\rightarrow}{ab},\overset {\rightarrow}{cd}\rangle
= -\langle\overset {\rightarrow}{ba},\overset {\rightarrow}{cd}\rangle$
and $\langle\overset {\rightarrow}{ax},\overset {\rightarrow}{cd}\rangle
+
\langle\overset {\rightarrow}{xb},\overset {\rightarrow}{cd}\rangle
=
\langle\overset {\rightarrow}{ab},\overset {\rightarrow}{cd}\rangle$
for all $a, b, c, d, x \in X$. We say that X satisfies the Cauchy-Schwarz inequality if
$\langle\overset {\rightarrow}{ab},\overset {\rightarrow}{cd}\rangle \leq d(a, b)d(c, d)$
for all $a, b, c, d \in X$. It is known (Corollary 3 of \cite{ber-nik2}) that a geodesically connected metric space
is a CAT(0) space if and only if it satisfies the Cauchy-Schwarz inequality.

A kind of convergence was introduced by Lim \cite{lim} in order to extend weak convergence in CAT(0) setting. Let $(X,d)$ be a Hadamard space, $\{x_n\}$ be a bounded sequence
in $X$ and $x\in X$. Let $r(x,\{x_n\})=\limsup d(x,x_n)$. The
asymptotic radius of $\{x_n\}$ is given by $r(\{x_n\})= \inf
\{r(x,\{x_n\})| x\in X \}$ and the asymptotic center of
$\{x_n\}$ is the set $A(\{x_n\})=\{x \in X|
r(x,\{x_n\})=r(\{x_n\})\}$. It is known that in a Hadamard space,
$A(\{x_n\})$ consists exactly one point.

\begin{definition}
	A sequence $\{x_n\}$ in a Hadamard space $(X,d)$
	$\triangle$-converges to $x\in X$
	if $A(\{x_{n_k}\})=\{x\}$, for each subsequence $\{x_{n_k}\}$ of $\{x_n\}$.\\
	We denote $\triangle$-convergence in $X$ by
	$\overset{\triangle}{\longrightarrow}$
	and the metric convergence by $\rightarrow$.
\end{definition}
It is well-known that every bounded sequence in a CAT(0) space has a $\Delta$-convergent subsequence (see \cite{kir}).

\begin{definition}
	Let $f:X\rightarrow]-\infty,+\infty]$. The domain of $f$ is defined by $D(f):=\{x\in X:\ f(x)<+\infty\}$. $f$ is proper iff $D(f)\neq\varnothing$.
\end{definition}

\begin{definition}
	A function $f:X\rightarrow]-\infty,+\infty]$ is called ($\Delta$-)lower semicontinuous (shortly, lsc) at $x\in D(f)$ iff $$\liminf_{n\rightarrow \infty}f(y_n)\geq f(x)$$
	for each sequence $y_n\rightarrow x$ ($y_n\overset{\triangle}{\longrightarrow}x$) as $n\rightarrow+\infty$. $f$ is called ($\Delta$-)lower semicontinuous iff it is ($\Delta$-)lower semicontinuous in each point of its domain. It is easy to see that every lower semicontinuous and quasi-convex function is $\Delta$-lower semicontinuous.
\end{definition}

\begin{definition}\cite{bac2}\label{scalar subdifferential}
	Let $f:X\rightarrow]-\infty,+\infty]$. For each $x\in D(f)$ the slop of $f$ at $x$ or scaler subdifferential of $f$ at $x$ is defined as follows
	$$|\partial f|(x):=\limsup_{y\rightarrow x}\frac{max\{f(x)-f(y), \ 0\}}{d(x,y)}$$
	$D(|\partial f|):=\{x\in X:\ |\partial f|(x)<+\infty\}$ is called the domain of $|\partial f|$. Obviously $D(|\partial f|)\subset D(f)$.
\end{definition}
 A point $x\in D(f)$ with $|\partial f|(x)=0$ is called a critical point of $f$. Obviously each local minimum of $f$ is a critical point. The converse is true if $f$ is pseudo-convex.
 \begin{proposition}
 	For each proper and pseudo-convex function $f:X\rightarrow]-\infty,+\infty]$, every critical point is a minimum point.
 \end{proposition}

 \begin{proof}
 	Suppose $x$ is a critical point which is not a minimum point. Then $f(x)>f(y)$ for some $y\in X$. By the definition of pseudo-convexity, there exist $\beta(x,y)>0$ and $0<\delta(x,y)\leq1$ such that $$f(x)-f(ty\oplus(1-t)x)\geq t\beta(x,y),\ \ \ \forall t\in (0,\delta(x,y))$$
 	Therefore
 	$$|\partial f|(x)=\limsup_{z\rightarrow x}\frac{\max\{f(x)-f(z),0\}}{d(z,x)}\geq \limsup_{t\rightarrow0}\frac{f(x)-f(ty\oplus(1-t)x)}{td(x,y)}\geq\frac{\beta(x,y)}{d(x,y)}>0,$$ which is a contradiction.
 \end{proof}
\begin{definition}
	Let $X$ be a metric space and $x_0\in X$. A function $f:X\rightarrow]-\infty,+\infty]$ is called coercive iff $$\lim_{d(x,x_0)\rightarrow+\infty}f(x)=+\infty.$$ Obviously coercivity is not dependent to point $x_0$ and every $\alpha$-strongly convex is coercive.
\end{definition}
\begin{definition}
	A metric space $X$ has property (R), iff the intersection of any decreasing family of nonempty, convex and closed sets is nonempty.
\end{definition}
\begin{proposition}\cite{bac2}\label{coercivity minimization}
	Let $f:X\rightarrow]-\infty,+\infty]$ be a quasi-convex and lsc function with $D(f)\neq\varnothing$, where $X$ is a geodesic metric space with property (R). If $f$ is coercive and bounded from blow,
	then $f$ has at least a minimum point.
\end{proposition}
\begin{proof}
	Set $C_n=\{x\in X: f(x)\leq {\rm inf}_{x\in X}f(x)+\frac{1}{n}\}$, then $C_n, n=1,2,\cdots$ are convex, closed, bounded and decreasing. By the assumption on $X$, $\cap_{n=1}^{\infty}C_n\neq\varnothing$. Therefore ${\rm Argmin} f\neq\varnothing$.
\end{proof}

\section{Proximal Point Algorithm}

 Throughout this section we assume $X$ is a Hadamard space, which is a unique geodesic metric space with property (R) (see \cite{kha-kha}). Using Proposition \ref{coercivity minimization}, we introduce the notion of resolvent for weakly convex and quasi-convex functions in CAT(0) spaces. Finally we study the proximal point algorithm for quasi-convex functions in non-positive curvature setting. Our results extend several results in the literature of  convex and quasi-convex functions in CAT(0) spaces and Hadamard manifolds (see \cite{ah-kh1, bac1, bac2, bfo, fer-oli, llmw, llm, pqo, tzh, tx, wlml}).

Theorem \ref{existence resolvent} and Proposition \ref{resolvent domain} were essentially mentioned in \cite{bac2} for convex functions.

\begin{theorem}\label{existence resolvent}
	Let $X$ be a CAT(0) space. Suppose $f:X\rightarrow]-\infty,+\infty]$ is $\alpha$-weakly convex. Then for each $x\in X$ and $\lambda<\frac{1}{2\alpha}$, the function $y\mapsto f(y)+\frac{1}{2\lambda}d^2(x,y)$ has a unique minimum point, which we denote it by $J_{\lambda}x$.
\end{theorem}

\begin{proof}
	Take $g(y)=f(y)+\frac{1}{2\lambda}d^2(x,y)$, then $g$ is $(\frac{1}{2\lambda}-\alpha)$-strongly convex. Then $g$ is coercive and strictly convex and therefore has a unique minimum point.
\end{proof}

\begin{proposition}\label{resolvent domain}
	Let $f:X\rightarrow]-\infty,+\infty]$ be $\alpha$-weakly convex and lsc. Then for each $\lambda>0$, $J_{\lambda}x\in D(|\partial f|)$ and consequently belongs to $D(f)$.
\end{proposition}

\begin{proof}
 By the definition of $J_{\lambda}x$, for each $y\in X$, $$f(J_{\lambda}x)-f(y)\leq\frac{1}{2\lambda}(d^2(y,x)-d^2(x,J_{\lambda}x))\leq\frac{1}{2\lambda}d(J_{\lambda}x,y)(d(y,x)+d(x,J_{\lambda}x)).$$
	Then we have
	$$0\leq\frac{max\{f(J_{\lambda}x)-f(y),0\}}{d(J_{\lambda}x,y)}\leq\frac{1}{2\lambda}(d(y,x)+d(x,J_{\lambda}x))$$
	Taking limsup when $y\rightarrow J_{\lambda}x$, we get
	$|\partial f|(J_{\lambda}x)<+\infty$ as desired.
\end{proof}

\begin{lemma} \label{projection resolvent}
	Suppose that $f:X\rightarrow]-\infty,+\infty]$ is $\alpha$-weakly convex, quasi-convex and lsc. Then for each $x\in X$ and each $z\in X$ such that $f(z)\leq f(J_{\lambda}x)$, $\langle\overset{\rightarrow}{J_{\lambda}xz},\overset{\rightarrow}{J_{\lambda}xx}\rangle\leq0$,
	equivalently $P_{\{z:\ f(z)\leq f(J_{\lambda}x)\}}x=J_{\lambda}x$.
\end{lemma}

\begin{proof}
	By Theorem \ref{existence resolvent} for each $x\in X$ and $\lambda<\frac{1}{2\alpha}$ $$f(J_{\lambda}x)+\frac{1}{2\lambda}d^2(x,J_{\lambda}x)\leq f(u)+\frac{1}{2\lambda}d^2(u,x),\ \ \forall u\in X$$
	Now for each $z\in X$ such that $f(z)\leq f(J_{\lambda}x)$ by taking $u=tJ_{\lambda}x\oplus(1-t)z$, and using quasi-convexity of $f$ and strong convexity of $d^2(\cdot,x)$, we get
	$$d^2(x,J_{\lambda}x)-d^2(z,x)+td^2(J_{\lambda}x,z)\leq0$$
	for each $0<t<1$. Now the result is concluded by letting $t\rightarrow1$ and the definition of quasi-inner product.
\end{proof}

The proximal point algorithm for each $\alpha$-weakly and quasi-convex function is defined by
\begin{equation}x_{n+1}=J_{\lambda_n}x_n\label{ppa}\end{equation}
Tanks to idea of Goudou and Munier \cite{gou-mun} to prove weak convergence of gradient flow of a quasi-convex function, we prove $\Delta$-convergence of the sequence given by \eqref{ppa} for continuous quasi-convex and weakly convex functions.
\begin{theorem}\label{weak convergence to critical point}
	Suppose $f:X\rightarrow]-\infty,+\infty]$ is a $\alpha$-weakly convex and quasi-convex function on a Hadamard space $X$, which is continuous in $D(f)$. If ${\rm Argmin} f\neq\varnothing$, then for each $\bar{\lambda}\leq\lambda_n\leq\frac{1}{2\alpha}$, the sequence given by \eqref{ppa} is $\Delta$-convergent to a critical point of $f$, which is a minimum point when $f$ is pseudo-convex.
\end{theorem}

\begin{proof}
By the hypothesis
\begin{equation}\label{argmin method}
f(x_n)+\frac{1}{2\lambda_{n-1}}d^2(x_n,x_{n-1})\leq f(y)+\frac{1}{2\lambda_{n-1}}d^2(y,x_{n-1}),\ \ \ \forall y\in X
 \end{equation}
	Taking $y=x_{n-1}$, we get: $\{f(x_n)\}$ is nonincreasing. Now taking $y=t\tilde{x}\oplus(1-t)x_n$ in (\ref{argmin method}), where $\tilde{x}\in {\rm Argmin} f$ and using quasi-convexity of $f$, we get
	 $$f(x_n)+\frac{1}{2\lambda_{n-1}}d^2(x_n,x_{n-1})\leq f(x_n)+\frac{1}{2\lambda_{n-1}}td^2(\tilde{x},x_{n-1})$$$$+\frac{1}{2\lambda_{n-1}}(1-t)d^2(x_n,x_{n-1})-\frac{1}{2\lambda_{n-1}}t(1-t)d^2(x_n,\tilde{x})$$
	 By letting $t\rightarrow0$, we receive to
	 $$d^2(x_n,x_{n-1})-d^2(x_{n-1},\tilde{x})+d^2(x_n,\tilde{x})\leq0.$$
	 Therefore $\{d(x_n,\tilde{x})\}$ is nonincreasing and $\langle\overrightarrow{x_n\tilde{x}},\overrightarrow{x_nx_{n-1}}\rangle\leq0$.
	In the sequel, we distinguish two cases:\\
	 1) $\liminf f(x_n)=\inf f=f(\tilde{x})$. Since $x_n$ is bounded, there is a subsequence $x_{n_j}$ of $x_n$, which is $\Delta$-convergent to $\bar{x}$. By $\Delta$-lower semi-continuity of $f$, which is concluded by lower semi-continuity and quasi-convexity of $f$, we get:
	 $$f(\bar{x})\leq\liminf f(x_{n_j})=\lim f(x_n)=f(\tilde{x}).$$
	 By Opials lemma (see Lemma 2.1 \cite{ran-kha}, also \cite{bss}) $x_n$ $\Delta$-converges to $\bar{x}\in {\rm Argmin} f$.\\
	 2) $\liminf f(x_n)=\lim f(x_n)> f(\tilde{x})=\inf_{x\in X}f$. By continuity of $f$ there exists $r>0$ such that for each $x\in B_r(\tilde{x})$, $f(x_n)>f(x),\ \ \forall n\geq n_0>0$.
	 By taking $x=x_{n-1},\  J_{\lambda_{n-1}}x=x_n,\  \text{and}\ z=x$ in Lemma \ref{projection resolvent}, we get
	 $$\langle\overrightarrow{x_nx},\overrightarrow{x_nx_{n-1}}\rangle\leq0,\ \ \ \forall x\in B_r(\tilde{x})$$
	 By quasi-inner product properties
	 $$\langle\overrightarrow{\tilde{x}x},\overrightarrow{x_nx_{n-1}}\rangle\leq \langle\overrightarrow{\tilde{x}x_n},\overrightarrow{x_nx_{n-1}}\rangle=\frac{1}{2}(d^2(x_{n-1},\tilde{x})-d^2(x_n,\tilde{x})-d^2(x_n,x_{n-1})).$$ By Proposition 9.2.28 of \cite{bac2} there exists a geodesic ray from $\tilde{x}$ parallel to geodesic segment $[x_n,x_{n-1}]$. Take the point $x\in B_r(\tilde{x})$ in this geodesic ray such that $d(x,\tilde{x})=\frac{r}{2}$, then
	 $$d(x,\tilde{x})d(x_n,x_{n-1})=\langle\overrightarrow{\tilde{x}x},\overrightarrow{x_nx_{n-1}}\rangle$$ Therefore $$rd(x_n,x_{n-1})\leq d^2(x_{n-1},\tilde{x})-d^2(x_n,\tilde{x}).$$
	 It proves that $x_n$ is a Cauchy sequence and therefore $x_n\rightarrow \bar{x}$. Now by letting $n\rightarrow+\infty$ from \eqref{argmin method}, we have
	 $$f(\bar{x})-f(y)\leq\frac{1}{2\bar{\lambda}}d^2(y,\bar{x})$$
	 Therefore
	 $$\frac{f(\bar{x})-f(y)}{d(y,\bar{x})}\leq\frac{1}{2\bar{\lambda}}d(y,\bar{x})$$
	 	Then
	 	$$0\leq\limsup_{y\rightarrow\bar{x}}\frac{max\{f(\bar{x})-f(y),0\}}{d(y,\bar{x})}\leq0,$$
	 	which follows $|\partial f|(\bar{x})=0$.
\end{proof}

\begin{corollary}\label{strong convergence to critical point}
	Suppose $f:X\rightarrow]-\infty,+\infty]$ is a $\alpha$-weakly convex and quasi-convex function on a Hadamard space $X$, which is continuous in $D(f)$. If ${\rm int}({\rm Argmin}\ f)\neq\varnothing$ and $\bar{\lambda}\leq\lambda\leq\frac{1}{2\alpha}$, then the sequence given by \eqref{ppa} is strongly convergent to a critical point of $f$, which is a minimum point of $f$ if $f$ is pseudo-convex.
\end{corollary}

\begin{proof}
Take $\tilde{x}\in {\rm Argmin}\ f$, since ${\rm int}({\rm Argmin}\ f)\neq\varnothing$ hence there exists $r>0$ such that $B_r(\tilde{x})\subseteq {\rm int}({\rm Argmin} \ f)$. Therefore for each $x\in B_r(\tilde{x})$, we have $f(x)\leq f(x_n)$ for all $n\in \mathbb{N}$. The rest of the proof is exactly similar to the proof of the second part of Theorem \ref{weak convergence to critical point}.
\end{proof}

We can reduce continuity assumption to lower semi continuity in Theorem \ref{weak convergence to critical point} but we obtain only $\Delta$-convergence to an element of $\{x\in X:\ f(x)\leq f(x_n),\ n\in \mathbb{N}\}$.

\begin{theorem}\label{delta-convergent}
	Suppose $f:X\rightarrow]-\infty,+\infty]$ is a proper, lsc, $\alpha$-weakly convex and quasi-convex function on a Hadamard space $X$ with ${\rm Argmin} f\neq\varnothing$. Then for each $\bar{\lambda}\leq\lambda_n\leq\frac{1}{2\alpha}$, the sequence given by \eqref{ppa} is $\Delta$-convergent to an element $\bar{x}$ of $\{x\in X:\ f(x)\leq f(x_n),\ n\in \mathbb{N}\}$.
\end{theorem}

\begin{proof}
	Since ${\rm Argmin} f\neq\varnothing$, $C:=\{x\in X:\ f(x)\leq f(x_n),\ n\in \mathbb{N}\}\neq\varnothing$. For each $z\in C$, since $f(z)\leq f(x_n)$ by Lemma \ref{projection resolvent}, $\langle \overrightarrow{x_nz},\overrightarrow{x_nx_{n-1}}\rangle\leq0$, therefore $d(x_n,z)\leq d(x_{n-1},z)$, for all $n=1, 2, \cdots$. So $\lim_{n\rightarrow+\infty}d(x_n,z)$ exists. On the other hand if the subsequence $x_{n_j}$ of $x_n$ $\Delta$-converges to $\bar{x}$, by $\Delta$-lower semi-continuity of $f$, we get
	$$f(\bar{x})\leq\liminf f(x_{n_j})=\lim f(x_n)$$ and since by \eqref{argmin method} $f(x_n)$ is nonincreasing, $\bar{x}\in C$. Now the result is concluded by Opials lemma in CAT(0) spaces \cite{ran-kha} (see also \cite{bss}).
	\end{proof}
\begin{theorem}
Suppose that the assumptions of Theorem \ref{delta-convergent} are satisfied. If  for some $n>0$, $\{x\in X:\ f(x)\leq f(x_n)\}$ is compact, then $\{x_n\}$ converges strongly to an element of $C$, which is also a critical point of $f$. Moreover, if $f$ is pseudo-convex, then $\{x_n\}$ converges strongly to a minimum point of $f$.
\end{theorem}

\begin{proof}
Note that $\{f(x_n)\}$ is nonincreasing. Since $\{x\in X:\ f(x)\leq f(x_n)\}$ is compact for some $n>0$, hence there are a subsequence $\{x_{n_k}\}$ of $\{x_n\}$ and $x^*\in C$ such that $x_{n_k}\rightarrow x^*$. Now, since $f$ is lower semi-continuous, we have:
$$f(x^*)\leq\liminf f(x_{n_k})=\lim f(x_n).$$
By Theorem \ref{delta-convergent} $x^*=\bar{x}\in C$, where $\bar{x}$ is the $\Delta$-limit of $x_n$ by Theorem \ref{delta-convergent}. Since By the proof of Theorem \ref{delta-convergent}, $\lim_{n\rightarrow+\infty}d(x_n,x^*)$ exists, hence $x_n\rightarrow x^*\in C$. By taking liminf from \eqref{argmin method}, we get
$$f(x^*)\leq f(y)+\frac{1}{2\bar{\lambda}}d^2(y,x^*)$$
which implies that $x^*$ is a critical point of $f$.
\end{proof}	

	The following lemma is elementary and we cancel the proof.
	
	\begin{lemma}\label{real sequence}
		Let $\{a_n\}$ and $\{b_n\}$ be two positive sequences such that $\{a_n\}$ is nonincreasing and convergent to zero. If $\sum_{n=1}^{\infty}a_nb_n<+\infty$, then $(\sum_{k=1}^nb_k)a_n\rightarrow0$ as $n\rightarrow+\infty$.
	\end{lemma}
	
\begin{theorem}
	Suppose $f:X\rightarrow]-\infty,+\infty]$ is proper and $\alpha$-strongly quasi-convex with ${\rm Argmin} f\neq\varnothing$. If $\sum_{n=0}^{\infty}\lambda_n=+\infty$, then the sequence $\{x_n\}$ generated by \eqref{ppa} converges strongly to the unique element $\tilde{x}$ of ${\rm Argmin} f$; moreover $d(x_n,\tilde{x})=o((\sum_{k=0}^n\lambda_k)^{-\frac{1}{2}})$.
\end{theorem}

\begin{proof}
	Let $\tilde{x}$ be the unique element of ${\rm Argmin} f$. By \eqref{ppa}, $\alpha$-strong quasi-convexity of $f$ and 1-strong convexity of $d^2(x,\cdot)$, we have
	$$f(x_n)+\frac{1}{2\lambda_{n-1}}d^2(x_n,x_{n-1})\leq f(t\tilde{x}\oplus(1-t)x_n)+\frac{1}{2\lambda_{n-1}}d^2(t\tilde{x}\oplus(1-t)x_n,x_{n-1})$$
	$$\leq f(x_n)-\alpha t(1-t)d^2(x_n,\tilde{x})+\frac{t}{2\lambda_{n-1}}d^2(x_{n-1},\tilde{x})+\frac{1-t}{2\lambda_{n-1}}d^2(x_n,x_{n-1})-\frac{t(1-t)}{2\lambda_{n-1}}d^2(x_n,\tilde{x})$$
	By substituting terms and letting $t\rightarrow0$, we get
	$$\alpha\lambda_{n-1}d^2(x_n,\tilde{x})\leq d^2(x_{n-1},\tilde{x})-d^2(x_n,\tilde{x}).$$
	Then summing up the recent inequality from $n=1$ to $\infty$, we arrive to
	\begin{equation}
	\alpha\sum_{n=1}^{\infty}\lambda_{n-1}d^2(x_n,\tilde{x})\leq d^2(x_0,\tilde{x})<+\infty\label{sum}
	\end{equation}
	Since $\sum_{n=0}^{\infty}\lambda_n=+\infty$, we get: $\liminf_{n\rightarrow+\infty}d^2(x_n,\tilde{x})=0$ and since $d^2(x_n,\tilde{x})$ is nonincreasing, we obtain the strong convergence of $x_n$ to $\tilde{x}$. The rate of convergence is concluded by \eqref{sum} and Lemma \ref{real sequence}.
\end{proof}

\end{document}